\documentclass[12pt]{article}
\usepackage{amssymb,amsmath,amsthm,epsfig}
\usepackage[margin=25mm]{geometry}
\usepackage{float}
\usepackage{sectsty}
\usepackage{lipsum}
\usepackage[labelfont=bf]{caption}
\allsectionsfont{\centering\MakeUppercase}
\newtheorem{thm}{Theorem}
\newtheorem{lem}{Lemma}

\newtheorem{cor}{Corollary}

\newtheorem{conj}{Conjecture}

\renewenvironment{proof}{{\bfseries Proof.}}{}

\makeatletter
\renewcommand{\maketitle}{\bgroup\setlength{\parindent}{0pt}
\begin{flushleft}
  \textbf{\@title}

  \@author
\end{flushleft}\egroup
}
\makeatother

\begin{document}
\title{\bf \Large A note on the augmented Zagreb index of cacti with fixed number of vertices and cycles}

\author{AKBAR ALI$^{\dag,\ddag,*}$ \& AKHLAQ AHMAD BHATTI$^{\dag}$\\
$^{\dag}$Department of Mathematics, National University of Computer \& Emerging Sciences, B-Block, Faisal Town, Lahore-Pakistan.\\
$^{\ddag}$Department of Mathematics, University of Gujrat, Gujrat-Pakistan.\\
$^{*}$Corresponding author: akbarali.maths@gmail.com}

\date{}
\maketitle

\renewcommand{\abstractname}{ABSTRACT}
\begin{abstract}
Let $\mathcal{C}_{n,k}$ be the family of all cacti with $k$ cycles and $n\geq4$ vertices. In the present note, the element of the class $\mathcal{C}_{n,k}$ having minimum augmented Zagreb index ($AZI$) is characterized. Moreover, some structural properties of the graph(s) having maximum $AZI$ value over the collection $\mathcal{C}_{n,0}$, are also reported.
\end{abstract}

\section*{Introduction}

All the graphs considered in the present study are simple, finite, undirected and connected.  The vertex set and edge set of a graph $G$ will be denoted by $V(G)$ and $E(G)$ respectively. Undefined notations and terminologies from (chemical) graph theory can be found in (Harary, 1969; Trinajsti\'{c}, 1992).

Topological indices are numerical parameters of a graph which are invariant under graph isomorphisms. There are many topological indices which are often used to model the physicochemical properties of chemical compounds in quantitative structure-property relation (QSPR) and quantitative structure-activity relation (QSAR) studies (Gutman \& Furtula, 2010; Trinajsti\'{c}, 1992). The atom-bond connectivity ($ABC$) index is one of such topological indices. The $ABC$ index was introduced by Estrada \textit{et al.} (1998) and it is defined as:
\[ABC(G)=\sum_{uv\in E(G)}\sqrt{\frac{d_{u}+d_{v}-2}{d_{u}d_{v}}},\]
where $d_{u}$ is the degree of the vertex $u\in V(G)$ and $uv$ is edge connecting the vertices $u$ and $v$. Details about this index can be found in the survey (Gutman \textit{et al.}, 2013), recent papers (Ashrafi \textit{et al.}, 2015; Dimitrov, 2014; Lin \textit{et al.}, 2015; Palacios, 2014; Raza, Bhatti \& Ali, 2015) and the related references cited therein.

Inspired by work on the $ABC$ index Furtula, Graovac \& Vuki\v{c}evi\'{c} (2010) proposed the following topological index and named it augmented Zagreb index ($AZI$):
\[AZI(G)=\displaystyle\sum_{uv\in E(G)}\left(\frac{d_{u}d_{v}}{d_{u}+d_{v}-2}\right)^{3}.\]
The prediction power of $AZI$ is better than $ABC$ index in the study of heat of formation for heptanes and octanes (Furtula, Graovac \& Vuki\v{c}evi\'{c}, 2010). Furtula, Gutman \& Dehmer (2013) undertook a comparative study of the structure-sensitivity of twelve topological indices by using trees and they concluded that $AZI$ has the greatest structure sensitivity. In the papers (Gutman \& To\v{s}ovi\'{c},  2013; Gutman, Furtula \& Elphick, 2014), the correlating ability of several topological indices was tested for the case of standard heats (enthalpy) of formation and normal boiling points of octane isomers, and it was found that the $AZI$ possess the best correlating ability among the examined topological indices. Hence it is meaningful to study the mathematical properties of $AZI$, especially bounds and characterization of the extremal elements of renowned graph families.

In (Furtula, Graovac \& Vuki\v{c}evi\'{c}, 2010), the extremal $n$-vertex chemical trees with respect to the $AZI$ were determined and the $n$-vertex tree having minimum $AZI$ value was characterized. Huang, Liu \& Gan (2012) gave various bounds on the $AZI$ for several families of connected graphs (e.g. chemical graphs, trees, unicyclic graphs, bicyclic graphs, etc.). Wang, Huang \& Liu (2012) established some new bounds on the $AZI$ of connected graphs and improved some results of the papers (Furtula, Graovac \& Vuki\v{c}evi\'{c}, 2010; Huang, Liu \& Gan, 2012). The present authors (Ali, Raza \& Bhatti, 2016) derived tight upper bounds for the $AZI$ of chemical bicyclic and unicyclic graphs. In (Ali, Bhatti \& Raza, 2016), the authors characterized the $n$-vertex graphs having the maximum $AZI$ value with fixed vertex connectivity (and matching number). Recently Zhan, Qiao \& Cai (2015) characterized the $n$-vertex unicyclic with the first and second minimal $AZI$ value, and $n$-vertex bicyclic graphs with the minimal $AZI$ value.

A graph $G$ is a cactus if and only if every edge of $G$ lies on at most one cycle. In recent years, many researchers studied the problem of characterizing the extremal cacti with respect to several well known topological indices over the class of all cacti with some fixed parameters. Some extremal results about the cacti can be found in the papers (Lu, Zhang \& Tian, 2006; Ali, Bhatti \& Raza, 2014; Chen, 2016; Du \textit{et al.}, 2015) and in the related references cited therein. In the present note, the cactus with minimum $AZI$ value is determined among all the cacti with fixed number of vertices and cycles. Moreover, some structural properties of the tree(s) having maximum $AZI$ value over the set of all trees with fixed number of vertices, are also reported.

\section*{The $AZI$ of Cacti}

Let $\mathcal{C}_{n,k}$ be the family of all cacti with $k$ cycles and $n\geq4$ vertices. As usual, denote by $S_{n}$ and $P_{n}$ the star graph and path graph (respectively) on $n$ vertices. A vertex of a graph is said to be pendent if it has degree 1. Let $G^{0}(n,k)\in\mathcal{C}_{n,k}$ be the cactus obtained from $S_{n}$ by adding $k$ mutually independent edges between the pendent vertices (see the Figure \ref{f1}). Note that $G^{0}(n,0)\cong S_{n}$.
\renewcommand{\figurename}{Fig.}
 \begin{figure}[H]
   \centering
    \includegraphics[width=3in, height=2in]{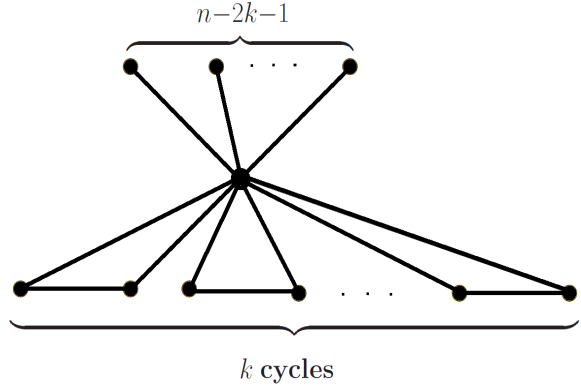}
    \caption{The cactus $G^{0}(n,k)$.}
     \label{f1}
      \end{figure}
Routine calculations show that
\[AZI(G^{0}(n,k))=(n-2k-1)\left(\frac{n-1}{n-2}\right)^{3}+24k.\]
Let us take
\[F(n,k)=(n-2k-1)\left(\frac{n-1}{n-2}\right)^{3}+24k.\]
Note that the collection $\mathcal{C}_{n,0}$ consist of all trees on $n$ vertices. Furtula, Graovac and Vuki$\check{c}$evi$\acute{c}$ characterized the $n$-vertex tree having minimum $AZI$ value:

\begin{lem}\label{L77}
(Furtula, Graovac \textsl{\&} Vuki\v{c}evi\'{c}, 2010) Let $T$ be any tree with $n\geq3$ vertices. Then
$$AZI(T)\geq F(n,0).$$
The equality holds if and only if $T\cong G^{0}(n,0)$.
\end{lem}

The class $\mathcal{C}_{n,1}$ consist of all unicyclic graphs on $n$ vertices. The following result about the characterization of unicyclic graph having minimum $AZI$ value is due to Huang, Liu and Gan:
\begin{lem}\label{L88}
(Huang, Liu \textsl{\&} Gan, 2012) Let $G$ be any unicyclic graph with $n\geq4$ vertices, then
$$AZI(G)\geq F(n,1).$$
The equality holds if and only if $G\cong G^{0}(n,1)$.
\end{lem}

The main result of this note will be proved with the help of following lemma:

\begin{lem}\label{L5}
For a fixed $p\geq1$, let
\[f(x,y)=\left(\frac{xy}{x+y-2}\right)^{3}-\left(\frac{(x-p)y}{x-p+y-2}\right)^{3},\]
where $x\geq2$, $x>p$ and $y\geq2$. Then the function $f(x,y)$ is increasing for $y$ in the interval $[2,\infty)$.
\end{lem}
\begin{proof}
Simple computations yield
\begin{equation}\label{Eq1}
\frac{\partial f(x,y)}{\partial y}=3y^{2}\left[\frac{x^{3}(x-2)}{(x+y-2)^{4}}-\frac{(x-p)^{3}(x-p-2)}{(x-p+y-2)^{4}}\right].
\end{equation}
Note that if $0<x-p\leq2$ then $\frac{\partial f(x,y)}{\partial y}$ is obviously positive and hence the lemma follows. Let us assume that $x-p>2$ and $g(x,y)=\frac{x^{3}(x-2)}{(x+y-2)^{4}}$. Then the first order partial derivative of $g(x,y)$ with respect to $x$, can be written as
\[\frac{\partial g(x,y)}{\partial x}=\frac{2x^{2}}{(x+y-2)^{5}}\left[x(2y-3)-3y+6\right].\]
It can be easily seen that the function $h(x,y)=x(2y-3)-3y+6$ is increasing in both $x$ and $y$ for $x,y\geq2$ and $h(2,2)$ is positive. Hence $\frac{\partial g(x,y)}{\partial x}$ is positive for all $x\geq2$ and $y\geq2$, which means that the function $g(x,y)$ is increasing in $x$. Therefore, from Eq. (\ref{Eq1}), the desired result follows.

\end{proof}

\vspace*{0.35cm}
The following elementary result will also be helpful in proving the main result of this note.

\begin{lem}\label{L6}
The function $f(x)=\left(\frac{x}{x-1}\right)^{3}$ is decreasing in the interval $[2,\infty)$.
\end{lem}

For a vertex $u$ of a graph $G$, denote by $N_{G}(u)$ (the neighborhood of $u$) the set of all vertices adjacent with $u$. Now we are in position to prove that $G^{0}(n,k)$ has the minimum $AZI$ value among all the cacti in the collection $\mathcal{C}_{n,k}$.

\begin{thm}\label{t1}
Let $G$ be any cactus belongs to the collection $\mathcal{C}_{n,k}$. Then
$$AZI(G)\geq F(n,k)$$
with equality if and only if $G\cong G^{0}(n,k)$.
\end{thm}
\begin{proof}
We will prove the theorem by double induction on $n$ and $k$. For $k=0$ and $k=1$, the result holds due to Lemma \ref{L77} and Lemma \ref{L88} respectively. Note that if $k\geq2$ then $n\geq5$. For $n=5$ there is only one cactus which is isomorphic to $G^{0}(5,2)$ and hence the theorem holds in this case. Let us assume that $G\in\mathcal{C}_{n,k}$ where $k\geq2$ and $n\geq6$. Then there are two possibilities:

\textit{Case 1.} If $G$ does not contain any pendent vertex. Then there exist three vertices $u,v$ and $w$ on some cycle of $G$ such that $u$ is adjacent with both the vertices $v,w$ where $d_{u}=d_{v}=2$ and $d_{w}=x\geq3$. Here we consider two subcases:\\
\textit{Subcase 1.1.} If there is no edge between $v$ and $w$. Then note that the graph $G'$ obtained from $G$ by removing the vertex $u$ and adding the edge $vw$, belongs to the collection $\mathcal{C}_{n-1,k}$. Hence by bearing in mind the Lemma \ref{L6}, inductive hypothesis and the fact $n\geq6$, one have
\begin{eqnarray*}
AZI(G)-F(n,k)&=& AZI(G')+8-F(n,k)\geq F(n-1,k)-F(n,k)+8\\
&=& (n-2k-1)\left[ \left(\frac{n-2}{n-3}\right)^{3}-\left(\frac{n-1}{n-2}\right)^{3}\right]\\
&&+\left[8-\left(\frac{n-2}{n-3}\right)^{3}\right]\geq8-\left(\frac{n-2}{n-3}\right)^{3}>0.
\end{eqnarray*}
\textit{Subcase 1.2.} If there is an edge between $v$ and $w$. Let $G^{*}$ be the graph obtained from $G$ by removing both the vertices $u,v$, then $G^{*}$ belongs to the class $\mathcal{C}_{n-2,k-1}$. Let $N_{G}(w)=\{u,v,u_{1},u_{2},...,u_{x-2}\}$. Then by virtue of Lemma \ref{L5}, Lemma \ref{L6} and inductive hypothesis, one have
\begin{eqnarray*}
AZI(G)-F(n,k)&=&AZI(G^{*})-F(n,k)+24\\
&+&\displaystyle\sum_{i=1}^{x-2}\left[ \left(\frac{xd_{u_{i}}}{x+d_{u_{i}}-2}\right)^{3}-\left(\frac{(x-2)d_{u_{i}}}{x+d_{u_{i}}-4}\right)^{3}\right]\\
&\geq& F(n-2,k-1)-F(n,k)+24\\
&=&(n-2k-1)  \left[ \left(\frac{n-3}{n-4}\right)^{3}-\left(\frac{n-1}{n-2}\right)^{3}\right]\geq0.
\end{eqnarray*}
The equality $AZI(G)-F(n,k)=0$ holds if and only if $G^{*}\cong G^{0}(n-2,k-1)$, $d_{u_{i}}=2$ for all $i$ (where $1\leq i\leq x-2$) and $n-2k-1=0$.

\textit{Case 2.} If $G$ has at least one pendent vertex. Let $u_{0}$ be the pendent vertex adjacent with a vertex $v_{0}$ and assume that $N_{G}(v_{0})=\{u_{0},u_{1},u_{2},...,u_{y-1}\}$. Without loss of generality one can suppose that $d_{u_{i}}=1$ for $0\leq i\leq p-1$ and $d_{u_{i}}\geq2$ for $p\leq i\leq y-1$. Let $G^{+}$ be the graph obtained from $G$ by removing the pendent vertices $u_{0},u_{1},u_{2},...,u_{p-1}$, then $G^{+}\in\mathcal{C}_{n-p,k}$ and hence one have

\begin{eqnarray*}
AZI(G)-F(n,k)&=&AZI(G^{+})+p\left(\frac{y}{y-1}\right)^{3}-F(n,k)\\
&&+\displaystyle\sum_{i=p}^{y-1}\left[ \left(\frac{yd_{u_{i}}}{y+d_{u_{i}}-2}\right)^{3}-\left(\frac{(y-p)d_{u_{i}}}{y+d_{u_{i}}-p-2}\right)^{3}\right].
\end{eqnarray*}
By virtu of Lemma \ref{L5} and inductive hypothesis, we have
\begin{equation}\label{ineq1}
AZI(G)-F(n,k)\geq F(n-p,k)-F(n,k)+p\left(\frac{y}{y-1}\right)^{3},
\end{equation}
with equality if and only if $G^{+}\cong G^{0}(n-p,k)$ and $d_{u_{i}}=2$ for all $i$ (where $p\leq i\leq y-1$). From Lemma \ref{L6} and inequality (\ref{ineq1}), it follows that
\[AZI(G)-F(n,k)\geq(n-p-2k-1)  \left[ \left(\frac{n-p-1}{n-p-2}\right)^{3}-\left(\frac{n-1}{n-2}\right)^{3}\right]\geq0,\]
with first equality if and only if $G^{+}\cong G^{0}(n-p,k)$ and $y=n-1$, and the last equality holds if and only if $n-p-2k-1=0$. This completes the proof.

\end{proof}

\vspace*{0.35cm}
Note that $\frac{\partial F(n,k)}{\partial k}=-2\left(\frac{n-1}{n-2}\right)^{3}+24$ is positive for all $n\geq4$, which means that $F(n,k)$ is increasing in $k$ (where $k\geq0$). Hence $F(n,k)$ attains it minimum value at $k=0$ and therefore from Theorem \ref{t1} one have:

\begin{cor}
If $G$ is any cactus with $n\geq4$ vertices, then
$$AZI(G)\geq F(n,0)$$
with equality if and only if $G\cong G^{0}(n,0)$.
\end{cor}

Now, we consider the problem of characterizing the graph(s) having maximum $AZI$ value over the collection $\mathcal{C}_{n,k}$. Let us start from the special case $k=0$, that is $\mathcal{C}_{n,0}=$ the family of all $n$-vertex trees. It can be easily checked that for $n=4,5,6$, the path $P_{n}$ has the maximum $AZI$ value in $\mathcal{C}_{n,0}$ and for $n=7,8,9,$ all those $n$-vertex trees in which every edge is incident with at least one vertex of degree 2, have maximum $AZI$ value in $\mathcal{C}_{n,0}$. Hence the graph with maximum $AZI$ value in the collection $\mathcal{C}_{n,0}$ needs not to be unique. Also, note that for the $n$-vertex tree $T^{+}$ depicted in Figure \ref{f2} one have $AZI(T^{+})=\left(\frac{9}{4}\right)^{3}+8(n-2)>8(n-1)=AZI(P_{n})$ for all $n\geq10$.
\renewcommand{\figurename}{Fig.}
 \begin{figure}[H]
   \centering
    \includegraphics[width=3.5in, height=1.2in]{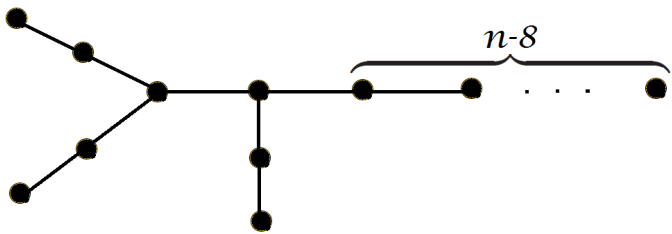}
    \caption{The $n$-vertex tree $T^{+}$ where $n\geq10$.}
     \label{f2}
      \end{figure}
Hence for $n\geq10$, the $n$-vertex tree with maximum $AZI$ value must be different from the path $P_{n}$. At this time, the problem of finding graph(s) with maximum $AZI$ value over the class $\mathcal{C}_{n,0}$ (and hence over the class $\mathcal{C}_{n,k}$) seems to be very hard and we leave it for future work. However, here we prove some results related to the structure of the $n$-vertex tree(s) having maximum $AZI$ value (given in Theorem \ref{t22}). From the derived results, it seems that the $n$-vertex tree with minimum $ABC$ index will actually have the maximum $AZI$ value. But, the problem of finding $n$-vertex tree with minimum $ABC$ index is still open and dozens of papers have been devoted to this open problem, see for example the survey (Gutman \textit{et al.}, 2013), articles (Dimitrov, 2014; Gutman, Furtula \& Ivanovi\'{c}, 2012; Ma \textit{et al.}, 2015) and the related references cited therein.

Now, let us prove some structural properties of the $n$-vertex tree(s) having maximum $AZI$ value. For this, we need the following elementary lemma.

\begin{lem}\label{L7}(Huang, Liu \textsl{\&} Gan, 2012)
If $\Psi(x,y)=\left(\frac{xy}{x+y-2}\right)^{3}$ then
\begin{enumerate}
  \item $\Psi(1,y)$ is decreasing for $y\geq2$.
  \item $\Psi(2,y)=8$ for all $y\geq2$.
  \item If $x\geq3$ is fixed, then $\Psi(x,y)$ is increasing for $y\geq3$.
\end{enumerate}
\end{lem}
Following Gutman, Furtula \& Ivanovi\'{c} (2012), we define some terms: A pendent vertex adjacent with a vertex having degree greater than 2 is called star-type pendent vertex. A path $v_{1}v_{2}...v_{l+1}$ of length $l$ in a graph $G$ is called pendent if one of the degrees $d_{v_{1}},d_{v_{l+1}}$ is 1 and other is greater than 2, and $d_{v_{i}}=2$ for all $i$ where $2\leq i\leq l$. A path $v_{1}v_{2}...v_{l+1}$ of length $l$ in a graph $G$ is called internal if both the degrees $d_{v_{1}},d_{v_{l+1}}$ are greater than 2 and $d_{v_{i}}=2$ for all $i$ where $2\leq i\leq l$.

\begin{thm}\label{t22}
For $n\geq10$, let $T^{*}\in\mathcal{C}_{n,0}$ be the tree with maximum $AZI$ value, then
\begin{enumerate}
  \item $T^{*}$ does not contain any internal path of length $l\geq2$.
  \item $T^{*}$ does not contain any pendent path of length $l\geq4$.
  \item $T^{*}$ contains at most one pendent path of length 3.
\end{enumerate}
\end{thm}

\begin{proof} (1). Suppose to the contrary that $T^{*}$ contains the internal path $v_{1}v_{2}...v_{l+1}$ of length $l\geq2$ where $d_{v_{1}},d_{v_{l+1}}\geq3$. We consider two cases:

\textit{Case 1.} If $T^{*}$ contains at least one star-type pendent vertex. Let $u$ be a star-type pendent vertex adjacent with a vertex $v$ (then $d_{v}\geq3$) and suppose that $T^{(1)}$ is the tree obtained from $T^{*}$ by moving the vertex $v_{2}$ on the edge $uv$ and adding the edge $v_{1}v_{3}$. Observe that both the trees $T^{(1)}$ and $T^{*}$ have same degree sequence. Hence, by virtue of Lemma \ref{L7} we have
\begin{eqnarray*}
AZI(T^{*})-AZI(T^{(1)})&=&\Psi(d_{v_{1}},d_{v_{2}})+\Psi(d_{v_{2}},d_{v_{3}})+\Psi(d_{u},d_{v})\\
&&-\Psi(d_{v_{1}},d_{v_{3}})-\Psi(d_{u},d_{v_{2}})-\Psi(d_{v_{2}},d_{v})\\
&=&\Psi(1,d_{v})-\Psi(d_{v_{1}},d_{v_{3}})<0,
\end{eqnarray*}
which is a contradiction to the maximality of $AZI(T^{*})$.

\textit{Case 2.} If $T^{*}$ does not contain any star-type pendent vertex. Suppose that $T^{(2)}$ is the tree obtained from $T^{*}$ by moving the vertex $v_{2}$ on any pendent edge and adding the edge $v_{1}v_{3}$. Then both the trees $T^{(2)}$ and $T^{*}$ have same degree sequence. If $l=2$, then by using Lemma \ref{L7}(2) and Lemma \ref{L7}(3), one have
\begin{eqnarray*}
AZI(T^{*})-AZI(T^{(2)})&=&\Psi(d_{v_{1}},d_{v_{2}})+\Psi(d_{v_{2}},d_{v_{3}})-\Psi(d_{v_{1}},d_{v_{3}})-8\\
&=&\Psi(d_{v_{1}},2)-\Psi(d_{v_{1}},d_{v_{3}})<0,
\end{eqnarray*}
which is a contradiction.\\
If $l\geq3$, then it can be easily checked that $AZI(T^{*})-AZI(T^{(2)})=0$. After repeating the above transformation sufficient number of times, one arrives at a tree ${T}^{(3)}$ such that $AZI(T^{*})-AZI(T^{(3)})<0$. This contradicts the maximality of $AZI(T^{*})$, which completes the proof of Part(1).

\vspace*{0.35cm}
(2). Let us suppose to the contrary that $T^{*}$ contains the pendent path $v_{1}v_{2}...v_{l+1}$ of length $l\geq4$ where $d_{v_{1}}=1$. We consider two cases:

\textit{Case 1.} If $T^{*}$ contains at least one star-type pendent vertex. Let $u$ be a star-type pendent vertex adjacent with a vertex $v$. Let $T^{(1)}$ be the tree obtained from $T^{*}$ by moving the vertex $v_{2}$ on the edge $uv$ and adding the edge $v_{1}v_{3}$. Note that both the trees $T^{(1)}$ and $T^{*}$ have same degree sequence. Hence, by bearing in mind the Lemma \ref{L7}(1), Lemma \ref{L7}(2) and the fact $d_{v}\geq3$ one have
\begin{eqnarray*}
AZI(T^{*})-AZI(T^{(1)})&=&\Psi(d_{v_{1}},d_{v_{2}})+\Psi(d_{v_{2}},d_{v_{3}})+\Psi(d_{u},d_{v})\\
&&-\Psi(d_{v_{1}},d_{v_{3}})-\Psi(d_{u},d_{v_{2}})-\Psi(d_{v_{2}},d_{v})\\
&=&\Psi(1,d_{v})-8<0,
\end{eqnarray*}
which is a contradiction to the maximality of $AZI(T^{*})$.

\textit{Case 2.} If $T^{*}$ does not contain any star-type pendent vertex. Let $w\in V(T^{*})$ has the maximum degree and $N_{T^{*}}(w)=\{w_{1},w_{2},...,w_{\Delta}\}$. Note that $d_{w_{i}}\geq2$ for all $i$ where $1\leq i\leq \Delta$.

\textit{Subcase 2.1} If at least one neighbor of $w$ has degree greater than 2. Suppose that $T^{(2)}$ is the tree obtained from $T^{*}$ by removing the edge $v_{2}v_{3}$ and adding the edge $v_{2}w$. Then by using Lemma \ref{L7}(2) and Lemma \ref{L7}(3), one have
\begin{eqnarray*}
AZI(T^{*})-AZI(T^{(2)})&=&\Psi(d_{v_{2}},d_{v_{3}})+\Psi(d_{v_{3}},d_{v_{4}})+\sum_{i=1}^{\Delta}\Psi(d_{w},d_{w_{i}})\\
&&-\Psi(d_{v_{2}},d_{w}+1)-\Psi(d_{v_{3}}-1,d_{v_{4}})-\sum_{i=1}^{\Delta}\Psi(d_{w}+1,d_{w_{i}})\\
&=&\sum_{i=1}^{\Delta}\left[\Psi(d_{w},d_{w_{i}})-\Psi(d_{w}+1,d_{w_{i}})\right]<0,
\end{eqnarray*}
which is again a contradiction.

\textit{Subcase 2.2} If $d_{w_{i}}=2$ for all $i$ where $1\leq i\leq \Delta$. Then Part(1) suggests that $T^{*}$ must be a starlike tree. Note that each branch of $T^{*}$ has length at least 2 and $w=v_{l+1}$ (in this case), hence $AZI(T^{*})=8(n-1)<\left(\frac{9}{4}\right)^{3}+8(n-2)=AZI(T^{+})$
where the tree $T^{+}$ is depicted in the Figure \ref{f2}. This is a contradiction to the maximality of $AZI(T^{*})$. This completes the proof of Part(2).

\vspace*{0.35cm}
(3). Suppose to the contrary that $T^{*}$ contains two pendent paths $v_{1}v_{2}v_{3}v_{4}$ and $u_{1}u_{2}u_{3}u_{4}$ of length three, where $d_{v_{1}}=d_{u_{1}}=1$. Then for the tree $T^{'}$ obtained from $T^{(*)}$ by deleting the edge $v_{1}v_{2}$ and adding the edge $v_{1}u_{1}$ one have $AZI(T^{*})=AZI(T^{'})$, which means that $T^{'}$ has also maximum $AZI$ value. But, $T^{'}$ contains a pendent path of length 4, namely $v_{1}u_{1}u_{2}u_{3}u_{4}$. This contradicts the Part(2).

\end{proof}

\vspace*{0.35cm}
Here it is worth mentioning that the properties given in the Theorem \ref{t22} also hold for the tree(s) having minimum $ABC$ index. We end this note with the following conjecture:

\begin{conj}
The $n$-vertex tree $T$ has the maximum $AZI$ value if and only if $T$ has the minimum $ABC$ index.
\end{conj}

\end{document}